\documentclass{article}
\usepackage{fullpage}
\usepackage{amssymb,amsmath,amsfonts,amsthm,mathrsfs}
\usepackage[all]{xy} \SelectTips{cm}{10}
\usepackage{color}

\numberwithin{equation}{section}

\newcommand{\FF}{\mathbb F}

\newcommand{\QQ}{\mathbb Q}
\newcommand{\RR}{\mathbb R}

\newcommand{\ZZ}{\mathbb Z} 
\newcommand{\Zhat}{\widehat\ZZ}

\newcommand{\OK}{\mathcal O_{K}}

\newcommand{\norm}[1]{ \left|\!\left| #1 \right|\!\right|  }

\def\cyc{{\operatorname{cyc}}}
\def\ab{{\operatorname{ab}}}
 
\def\Gal{\operatorname{Gal}}
 
\def \ML {\operatorname{M}}  
\def \GL {\operatorname{GL}}  

\def \SL {\operatorname{SL}}
\def \PSL {\operatorname{PSL}}
\def\Aut{\operatorname{Aut}} 
\def\End{\operatorname{End}}
\def\Frob{\operatorname{Frob}}
\newcommand{\tor}{{\operatorname{tor}}}

\def\id{\operatorname{id}}

\def\tr{\operatorname{tr}}

\newcommand{\surjects}{\twoheadrightarrow}
\newcommand{\injects}{\hookrightarrow}
\newcommand{\isom}{\simeq}

\newcommand{\intersect}{\cap} 
\newcommand{\Union}{\bigcup} 
\newcommand{\union}{\cup} 

\newcommand{\Kbar}{\overline{K}}

\newcommand{\kvbar}{\overline{k_v}}

\newcommand{\Kcyc}{K^\cyc}

\newcommand{\GalK}{{\Gal}(\Kbar/K)}

\newcommand{\Galkv}{{\Gal}(\kvbar/k_v)}
\newcommand{\linfty}{\ell^{\infty}}

\newcommand{\GalKcyc}{{\Gal}(K^{\cyc}/K)}

\def\Etor{E(\Kbar)_{\tor}}
\def\El{E[\ell]}

\def\Gab{\{\pm 1\}\times\hat{\ZZ}^{*}}
\def\Gl{\GL_2(\ZZ_{\ell})}
\def\Gln{\GL_2(\ZZ/\ell^n\ZZ)}
\def\Gfl{\GL_2(\FF_{\ell})}
\def\GZ{\GL_2(\hat{\ZZ})}

\def\Gm{\GL_2(\ZZ/m\ZZ)}

\def\Gprod{\prod_{\ell \ \text{prime}}\Gl}
\def\Sl{\SL_2(\ZZ_{\ell})}

\def\Zl{\ZZ_{\ell}}
\def\Fl{\FF_{\ell}}

\def\Zhat{\hat{\ZZ}}

\def\Zm{\ZZ/m\ZZ}

\def\sgn{\mathop{\rm sgn}\nolimits}

\def\Qalpha{\QQ(\alpha)}
\def\OK{\mathcal{O}_K}

\newtheorem{theorem}{Theorem}[section]
\newtheorem{lemma}[theorem]{Lemma}
\newtheorem{corollary}[theorem]{Corollary}
\newtheorem{proposition}[theorem]{Proposition}

\theoremstyle{definition}
\newtheorem{definition}[theorem]{Definition}

\theoremstyle{remark}
\newtheorem{remark}[theorem]{Remark}

\definecolor{webcolor}{rgb}{0.8,0,0.2}
\definecolor{webbrown}{rgb}{.6,0,0}
\usepackage[
        colorlinks,
        linkcolor=webbrown,  filecolor=webcolor,  citecolor=webbrown, 
        backref,
        pdfauthor={Aaron Greicius}, 
       pdftitle={Elliptic curves with surjective adelic Galois representations},
]{hyperref}

\begin{document}
\title{Elliptic curves with surjective adelic Galois representations}

\author{Aaron Greicius\\ Institut f\"ur Mathematic\\ Humboldt-Universit\"at zu Berlin\\ 12489 Berlin, Germany\\ greicius@math.hu-berlin.de\\ http://www.mathematik.hu-berlin.de/\~{}greicius}
\date{}
\maketitle

\begin{abstract}Let $K$ be a number field. The $\GalK$-action on the the torsion of an elliptic curve $E/K$ gives rise to an adelic representation $\rho_E\colon\GalK\rightarrow\GZ$. From an analysis of maximal closed subgroups of $\GZ$ we derive useful necessary and sufficient conditions for $\rho_E$ to be surjective. Using these conditions, we compute an example of a number field $K$ and an elliptic curve $E/K$ that admits a surjective adelic Galois representation.  
\end{abstract}

\section{Introduction} \label{S:intro}
Let $E/K$ be an elliptic curve, with $K$ a number field. Fix an algebraic closure $\Kbar$ of $K$ and define $G_{K}:=\GalK$. For each positive integer $m\geq 1$ and each prime number $\ell\geq 1$, the action of $G_{K}$ on the various torsion subgroups of $E(\Kbar)$ gives rise to continuous representations
\[ \rho_{E,m}\colon G_K\rightarrow\Aut(E(\Kbar)[m])\simeq\Gm\]
\[\rho_{E,\linfty}\colon G_K\rightarrow\Aut(E(\Kbar)[\linfty])\isom\Gl.\]
These representations are neatly packaged into the single representation  
\[\rho_{E}\colon G_K\rightarrow\Aut(\Etor)\isom\GZ\]
describing the action of $G_K$ on the full torsion subgroup of $E(\Kbar)$. Here $\Zhat:=\varprojlim\Zm\isom\prod_{\ell \ \text{prime}}\Zl$ is the profinite completion of $\ZZ$. We refer to $\rho_{E,\linfty}$ and $\rho_{E}$ respectively as the \emph{$\ell$-adic} and \emph{adelic} representations associated to $E/K$. Serre proves in \cite{Se72} that if $E$ does not have complex multiplication (non-CM), then the adelic image of Galois, $\rho_{E}(G_{K})$, is open in $\GZ$. Equivalently, since the adelic image is always a closed subgroup, Serre's result asserts that $\rho_{E}(G_{K})$ is of finite index in $\GZ$ when $E/K$ is non-CM. The question naturally arises then, whether this index is ever 1. In other words, are there elliptic curves $E/K$ for which $\rho_{E}$ is surjective? 

When $K=\QQ$ the answer is `no', as Serre himself proves in the same paper (\cite[\S 4.4]{Se72}). As we show below, the obstacle in this situation is essentially the fact that $\QQ^\cyc=\QQ^\ab$, leaving open the possibility of $\rho_{E}$ being surjective for other number fields $K$. Indeed, we provide simple necessary and sufficient conditions for the adelic representation to be surjective and give an example of a (non-Galois) cubic extension $K/\QQ$ and an elliptic curve $E/K$ for which $\rho_{E}$ is surjective.
\subsection{Statement of results}\label{SS:results} 
When is $\rho_{E}$ surjective; that is, when do we have $\rho_{E}(G_K)=\GZ$? We may put aside the arithmo-geometric component of this question for the time being and ask more generally: When is a closed subgroup $H\subseteq\GZ$ in fact all of $\GZ$? 

The group $\GZ$ is both a profinite and a product group, as articulated by the two isomorphisms
\begin{equation}\label{E:1}
\varprojlim \Gm \simeq\GZ\simeq\prod_{\ell \ \text{prime}}\Gl.\end{equation}
Consider the projection maps $\pi_{\ell}\colon\GZ\rightarrow\Gl$ that arise from the product group description of $\GZ$. An obvious necessary condition for a closed subgroup $H$ to be all of $\GZ$ is that the restrictions $\pi_\ell\colon H\rightarrow\Gl$ must all be surjective. It turns out that this condition is not so far from being sufficient; one need only further stipulate that the restriction of the abelianization map to $H$ be surjective. As we will show, the abelianization of $\GZ$ is isomorphic to $\Gab$, and we may describe the abelianization map as $(\sgn,\det)\colon\GZ\rightarrow\Gab$, where $\det$ is the determinant map, and $\sgn\colon\GZ\rightarrow\{\pm 1\}$ is a certain `sign' map on $\GZ$. Taken together this yields the following theorem.
\begin{theorem}
\label{T:H=G}
Let $H\subseteq \GZ$ be a closed subgroup. Then $H=\GZ$ if and only if
\begin{itemize}
\item[(i)]$\pi_\ell: H\rightarrow \Gl$ is surjective for all 
primes $\ell$ and
\item[(ii)]$(\sgn,\det):H\rightarrow \Gab$ is 
surjective.
\end{itemize}
\end{theorem}
Returning to our representation $\rho_{E}$, we can easily rephrase Theorem~\ref{T:H=G} to derive simple necessary and sufficient conditions for surjectivity. 
\begin{theorem}\label{T:rhoSurj}
Let $E/K$ be an elliptic curve defined over a 
number 
field 
$K$. Let $\Delta\in K^\times$ be the discriminant of any Weierstrass model of $E/K$. 
Then $\rho_{E}$ is surjective if and only if
\begin{itemize}
\item[(i)] the $\ell$-adic representation $\rho_{\linfty}\colon G_K\rightarrow\Gl$ is surjective for all $\ell$,
\item[(ii)]$K\cap\QQ^{\cyc}=\QQ$ and
\item[(iii)] $\sqrt{\Delta}\notin K^{\cyc}$.
\end{itemize}
\end{theorem}
\begin{remark}Suppose $\Delta$ and $\Delta'$ are the discriminants of two Weierstrass models of $E/K$. Then $\Delta'=u^{12}\Delta$ for some $u\in K$. Thus $\Delta\notin\Kcyc$ if and only if $\Delta'\notin\Kcyc$. In other words, condition (iii) is well-defined.
\end{remark}
\begin{remark}
Condition (i) is clearly equivalent to the surjectivity of the restrictions of the projection maps $\pi_\ell$ to $\rho_{E}(G_K)$. As will be explained below, conditions (ii) and (iii) are equivalent to the surjectivity of the restriction of the abelianization map to $\rho_{E}(G_K)$.
\end{remark}

The theorem suggests that when on the hunt for an elliptic curve with surjective adelic Galois representation, we should first find a ``suitable'' extension $K/\QQ$ which satisfies condition (ii) and which could possibly satisfy condition (iii) for some $E/K$. Note first that for $K=\QQ$, condition (iii) will never be satisfied, as  $\sqrt{\Delta}\in\QQ^\ab=\QQ^\cyc$. Thus there are no elliptic curves $E/\QQ$ with surjective $\rho_{E}$. Likewise, condition (ii) will not be satisfied by any quadratic extension of $\QQ$. With an eye toward finding a candidate number field of minimal degree, we should then cast our net among the non-Galois cubic extensions of $\QQ$. Having fixed a candidate number field $K$, the more difficult task is finding an elliptic curve $E/K$ satisfying condition (i). In our example we work over the field $\QQ(\alpha)$, where $\alpha$ is the real root of $f(x)=x^3+x+1$. Thanks to similarities between the field $\Qalpha$ and $\QQ$, we are able to extend to elliptic curves $E/\Qalpha$ the techniques Serre uses in \cite{Se72} to compute the $\ell$-adic images of elliptic curves $E/\QQ$. This allows us to easily find examples of elliptic curves over $\Qalpha$ with surjective adelic Galois representations. We record one example here as a theorem.
\begin{theorem}\label{T:example}
Let $K=\QQ(\alpha)$, where $\alpha$ is the real root of $f(x)=x^3+x+1$. Let $E/K$ be the elliptic curve defined by the Weierstrass equation $y^2+2xy+\alpha y=x^3-x^2$. The associated adelic representation $\rho_{E}\colon G_K\rightarrow\GZ$ is surjective.
\end{theorem}

\subsection{Related results}The results of this paper first appeared in my doctoral thesis (\cite{GThesis}), wherein I also asked, in the spirit of Duke's \cite{Duke} and N.~Jones' \cite{Jones}, whether in fact for any suitable $K$ ``most'' elliptic curves have surjective adelic Galois representations. David Zywina has since answered this question in the affirmative. 

In more detail, given a number field $K$ with ring of integers $\mathcal{O}_K$, fix a norm $\norm{\cdot}$ on $\RR\otimes_\ZZ\mathcal{O}_K^2\isom\RR^{2[K:\QQ]}$. Given $x>0$, define $B_K(x)$ to be the set of pairs $(a,b)\in\mathcal{O}_K^2$ having norm no greater than $x$ for which the associated curve $E(a,b)$ given by $y^2=x^3+ax+b$ is an elliptic curve. Now define $S_K(x)$ to be the subset of $B_K(x)$ consisting of pairs $(a,b)$ whose associated elliptic curves have surjective adelic Galois representations. In \cite{Zywina} Zywina proves the following theorem using sieve methods.
\begin{theorem}[Zywina] Suppose $K\ne\QQ$ satisfies $K\cap\QQ^\cyc=\QQ$. Then 
\[
\lim_{x\to \infty} \frac{ |S_K(x)| }{ |B_K(x)| } = 1.
\]
In other words, most elliptic curves over $K$ have surjective adelic Galois representation.
\end{theorem}
\begin{remark}In fact Zywina considers more generally the situation where $K\cap\QQ^\cyc$ is not required to be $\QQ$. As we recall below, in terms of arithmetic this means simply that the inclusion $\det(\rho_E(G_K))\subseteq \Zhat^*$ is not necessarily an equality. Zywina proves (\cite[Th. 1.3]{Zywina}) the expected generalization to this setting; namely, if $K\ne\QQ$, then for ``most'' elliptic curves $E/K$ we have $\rho_E(G_K)=\{A\in\GZ\colon \det A\in\det(\rho_E(G_K))\}$. \end{remark}
\subsection{Notation and conventions}\label{SS:Notations}
Let $G$ be a topological group, and let $H\subseteq G$ be a
closed subgroup.
The {\it commutator of $H$}, denoted $H'$, is the closure of the usual commutator
subgroup $[H,H]$. By a quotient of $G$ we shall always mean a continuous quotient. The \emph{abelianization} of $G$ is the quotient $G^\ab:=G/G'$. 

The two isomorphisms of Equation~\ref{E:1} give rise to reduction maps $r_m:\GZ\rightarrow \Gm$ and projection maps $\pi_\ell: \GZ\rightarrow
\Gl$, respectively. Following \cite{LT}, we associate with these
maps the following notation:
\begin{itemize}
\item[(i)] Let $P\subset\ZZ$ be the set of prime numbers. Given any $S\subseteq P$ let $\pi_S$ be the projection $\pi_S:\GZ\rightarrow\prod_{\ell\in S}\Gl$. Furthermore,  for any $X\subseteq \GZ$ we define
$X_S:=\pi_S(X)$. If $S=\{\ell\}$, we write $X_\ell$ instead of $X_{\{\ell\}}$. Thus, if we
let $G=\GZ$, then under our notation we have $G_\ell=\Gl$ and $G_S=\prod_{\ell\in S}\Gl$;

\item[(ii)]Similarly, given any nonnegative integer $m$ and any subset
$X\subseteq \GZ$, we define $X(m)=r_m(X)\subseteq \Gm$.
\end{itemize}
As a slight abuse, we will use the same notation when working with subgroups of $\Gl$ or $\Gm$.

Let $K$ be a number field with algebraic closure $\Kbar$. We set $G_K:=\GalK$. The set of finite places of $K$ will be denoted $\Sigma_K$. For a rational prime $\ell$, let $S_\ell$ be the set of places of $\Sigma_K$ lying above $\ell$. Next, define $\Sigma_{\Kbar}$ to be the inverse limit of the sets $\Sigma_{K'}$, where $K'$ runs over the finite subextensions of $\Kbar/K$.  Fix a place $v\in\Sigma_K$. The completion at $v$ is denoted by $K_v$,  the residue field at $v$ by $k_v$, and the cardinality of the residue field by $N_v$.  We define $S_v:=\{w\in\Sigma_{\Kbar}\colon w\mid v\}$.  Given $w\in S_v$, the \emph{decomposition group} of $w$ is defined as $D_w:=\{\sigma\in G_K\colon \sigma(w)=w\}$. There is a surjection $D_w\surjects\Galkv$. The kernel of this map is the \emph{inertia group} of $w$, denoted $I_w$. The Frobenius element $\Frob_w$ is the coset of $D_w/I_w$ mapping to the Frobenius element of $\Galkv$. A Galois representation $\rho$ is \emph{unramified at $v$} if $I_w\subseteq\ker\rho$ for some (and hence all) $w\in S_v$. 

Lastly, if $E/K$ is an elliptic curve, we define $S_E$ to be the set of places in $\Sigma_K$ where $E$ has bad reduction.
\subsection*{Acknowledgements} Thanks are due to Bjorn Poonen for suggesting this problem to me, to David Zywina for useful discussions, and to the referee, who suggested phrasing Proposition~\ref{semistableborel} and its corollaries in their present, more general form. This work was partially supported by German Research Society (DFG) grant GR 3448/2-1.

\section{Some (profinite) group theory} In this section we set about proving Theorem~\ref{T:H=G}. As we shall see, every proper closed subgroup $H$ of a profinite group $G$ is contained in a maximal closed subgroup, from which it follows that $H=G$ if and only if $H$ is not contained in any maximal closed subgroup. The necessary and sufficient conditions described in Theorem~\ref{T:H=G} are then a consequence of Proposition~\ref{P:Max} below, which describes the maximal closed subgroups of $\GZ$ in terms of the quotient maps to $\Gl$ and $\GZ^{\ab}$. 

\subsection{Maximal closed subgroups} 

\begin{definition} Let $G$ be a topological group.  A {\it maximal closed 
subgroup
of $G$} is a
closed subgroup $H\subsetneq G$ such that if $K$ is closed and $H\subseteq K\subsetneq G$, then $H=K$.
\end{definition}

\begin{lemma} Let $G$ be a profinite group. Any closed subgroup $H\subsetneq G$ is
contained in a maximal closed subgroup. All maximal closed subgroups of $G$ are open.
\end{lemma}
\begin{proof} Let $H$ be any proper closed subgroup of $G$. Since $G$ is profinite, we have $H=\overline{H}=\bigcap\{HN\vert 
N\triangleleft_o
G\}$
( see \cite[0.3.3]{Wilson}). Here $N\triangleleft_o G$ signifies that $N$ is a normal open 
subgroup of $G$. If $HN=G$
for all $N\triangleleft_o G$, then $H=G$, a contradiction. Thus there is a
$N\triangleleft_o G$
such that $H\subseteq HN\subsetneq G$. Now consider the quotient map $\pi\colon G\rightarrow G/N$. Since $N$ is open, the quotient group $G/N$ is finite. Since $HN/N\subsetneq G/N$, there is a maximal subgroup $K\subsetneq G/N$ containing $HN/N$. Then $L=\pi^{-1}(K)$ is a maximal closed subgroup of $G$ containing $HN$, and hence $H$. In fact $L$ is open, since $[G:L]$ is finite. Thus we have proved that every proper closed subgroup is contained in an {\it open} maximal closed subgroup. It follows that maximal closed subgroups are themselves open. 
\end{proof}

Consider now a product of profinite groups $G=\prod_{\alpha\in\Lambda}G_{\alpha}$.  As the projections $\pi_{\alpha}$ are all surjective, we get many maximal closed subgroups of $G$ of the form $\pi_{\alpha}^{-1}(K_{\alpha})$, where $K_{\alpha}\subsetneq G_{\alpha}$ is a maximal closed subgroup of $G_{\alpha}$. Similarly, there are maximal closed subgroups of $G$ arising from the abelianization $G^{\ab}=G/G'$ via the abelianization map $G\rightarrow G/G'$. We show below that under certain technical conditions all maximal closed subgroups of $G$ are accounted for in this way. We will make use of the following notion.   
\begin{definition}Given a profinite group $G$, let $\text{Quo}(G)$ be the set
of isomorphism
classes of finite, nonabelian, simple quotients of $G$.
\end{definition}
\begin{remark} In \cite[IV-25]{Se68} Serre similarly defines $\text{Occ}(G)$ to be the set of (isomorphism classes of) finite nonabelian simple groups $H$ that ``occur'' in $G$, in the sense that there exist closed subgroups $K_{1}\subseteq K_{2}\subseteq G$ with $K_{1}\triangleleft K_{2}$ and $K_{2}/K_{1}\isom H$. We have $\text{Quo}(G)\subseteq\text{Occ}(G)$.  As with Serre's $\text{Occ}$, the operation $\text{Quo}$ behaves well with respect to inverse limits. Namely, If $G=\varprojlim G_{\alpha}$ is an inverse limit of profinite groups, and the maps $G\rightarrow G_{\alpha}$ are all surjective, then $\text{Quo}(G)=\Union_{\alpha\in\Lambda}\text{Quo}(G_{\alpha})$. In particular $\text{Quo}(\prod_{\alpha}G_{\alpha})=\Union\text{Quo}(G_{\alpha})$.
\end{remark}
\begin{proposition}\label{P:Max}
Let $\{G_{\alpha}\}_{\alpha\in\Lambda}$ be a family of profinite groups such that $\text{Quo}(G_{\alpha})\cap\text{Quo}(G_{\alpha'})=\emptyset$ for all $\alpha\ne\alpha'$. Let $G=\prod_{\alpha\in\Lambda}G_{\alpha}$ and suppose $H\subsetneq G$ is a maximal closed subgroup.  Then
either
\begin{itemize}
\item[(i)] $H_{\alpha}=\pi_{\alpha}(H)$ is a maximal closed subgroup of $G_\alpha$ for some
$\alpha$, in which case $H=H_{\alpha}\times\prod_{\alpha'\ne \alpha}G_\alpha$, or

\item[(ii)] $H_{\alpha}=G_\alpha$ for all $\alpha$, in which case $H$ contains $G'$ and the
image of $H$ in $G^{\ab}=G/G'$ is maximal.
\end{itemize}
In other words, all maximal closed subgroups of $G$ arise either from a maximal closed subgroup of $G_{\alpha}$ for some $\alpha\in\Lambda$, or from a maximal closed subgroup of $G^{\ab}=G/G'$. 
\end{proposition}
The proof of Proposition~\ref{P:Max} will rely on the following variant of Goursat's Lemma.

\begin{lemma}[Topological Goursat's Lemma]Let $G_1,G_2$ be profinite groups, and let $H$ be a maximal closed
subgroup of
$G_1\times G_2$, such that $\pi_i(H)=G_i$ for the two projections $\pi_1$ and
$\pi_2$.
Identifying the $G_i$ with their canonical injections in $G_1\times G_2$, let
$N_i=H\cap G_i$.
Then the $N_i$ are open, normal subgroups of the $G_i$, the quotients $G_i/N_i$ 
are 
simple 
groups, and there is an
isomorphism $\phi:G_1/N_1\simeq G_2/N_2$, whose graph is induced by
$H$.
\end{lemma}
\begin{proof}The proof that the $N_i$ are open and normal is straightforward. 
The
isomorphism
$\phi$ then arises from the chain of isomorphisms $G_1/N_1\simeq H/N_1N_2\simeq
G_2/N_2$.

It remains only to show that the $G_i/N_i$ are simple. The isomorphism $\phi$ implies that $N_1=G_1$ if and only if $N_2=G_2$ if and only
if $H=G_1\times G_2$. Since $H$ is maximal, we see that $N_1\ne G_1$. Now 
suppose
we had
$N_1\subsetneq N\subsetneq G_1$ for some normal subgroup $N\triangleleft G_1$.
Since $N$ is
closed and normal in $G_1$, it is also closed and normal considered as a subgroup of $G_1\times G_2$, in which case $HN$ is
closed and
$H\subsetneq HN$. Furthermore $HN\subsetneq G_{1}\times G_{2}$, since $HN\cap G_1=(H\cap G_1)N=N_1N=N\ne G_1$. This contradicts the fact that $H$ is maximal. Thus there can be no such $N$. This proves that $G_1/N_1$ (and 
hence 
$G_2/N_2$) is simple.
\end{proof}

\begin{proof}[Proof of Proposition~\ref{P:Max}]

If $H_{\alpha}\subsetneq G_{\alpha}$ for some $\alpha$, then $H_{\alpha}$ is maximal in $G_{\alpha}$. Furthermore, since  $H\subseteq H_{\alpha}\times\prod_{\alpha'\ne\alpha}G_{\alpha}\subsetneq G$, we must have  $H=H_{\alpha}\times\prod_{\alpha'\ne\alpha}G_{\alpha}$. 

Assume now that $H_{\alpha}=G_{\alpha}$ for all $\alpha\in\Lambda$. Since $H\subsetneq G$ is open, there is a finite nonempty set $S\subseteq\Lambda$ such that $\ker\pi_{S}\subseteq H$. Since $H$ is maximal, the projection $H_{S}$ is a maximal closed subgroup of $G_{S}$ and $H=H_{S}\times \prod_{\alpha'\notin S}G_{\alpha'}$. As
$G'=\prod_{\alpha\in\Lambda}G_{\alpha}'$, it suffices to prove the corresponding statement for 
$H_{S}$. In 
other words, we need only prove that given any finite set $S\subseteq\Lambda$ and any 
maximal
closed subgroup $H\subseteq G_{S}$,
if $H_{\alpha}=G_{\alpha}$ for all $\alpha\in S$, then
$G_{S}'\subseteq H$.  
We do so using induction on $\vert S \vert$, the case $\vert S\vert=1$ being
trivial.

Assume $\vert S\vert>1$.  Take any $\alpha\in S$ and set $S'=S-\{\alpha\}$.

Suppose $H_{S'}\ne G_{S'}$. Then $H_{S'}$ is maximal and we have
$H=H_{S'}\times G_\alpha$. By induction, $H_{S'}$ contains $G_{S'}'$, and thus $H$
contains $G_{S}'$.

Suppose $H_{S'}=G_{S'}$. Let $N_{S'}=H\cap G_{S'}$ and let
$N_{\alpha}=H\cap G_\alpha$, where we identify $G_{\alpha}$ with $\ker\pi_{S'}$ and $G_{S'}$ with $\ker\pi_{\alpha}$. By the Topological Goursat's Lemma these subgroups are normal
in $G_S$ and there is an isomorphism of simple groups $G_{S'}/N_{S'}\simeq G_\alpha/N_\alpha$. But 
\begin{align*}
\text{Quo}(G_{S'})\cap\text{Quo}(G_{\alpha})&=\text{Quo}(\prod_{\alpha'\in S'}(G_{\alpha'}))\cap\text{Quo}(G_{\alpha})\\
&=\Union_{\alpha'\in S'}\text{Quo}(G_{\alpha'})\cap\text{Quo}(G_{\alpha})\\
&=\emptyset.
\end{align*}
Thus the simple groups $G_{S'}/N_{S'}$ and $G_{\alpha}/N_{\alpha}$ are abelian, in which case $G_{S'}'\subseteq N_{S'}$ and $G_{\alpha}'\subseteq N_{\alpha}$. It follows that $G_{S}'\subseteq H$.
\end{proof}
\begin{corollary}\label{C:Max} Let a $H$ be a maximal closed subgroup of $\GZ=\prod_{\text{$\ell$ prime}}\Gl$. Then either
\begin{itemize}
\item[(i)] $H_{\ell}=\pi_{\ell}(H)$ is a maximal closed subgroup of $\Gl$ for some prime $\ell$ or
\item[(ii)] $H_{\ell}=\Gl$ for all $\ell$, in which case $G'\subseteq H$.
\end{itemize}
\end{corollary}
\begin{proof}We need only show that the groups $\Gl$ satisfy the technical condition of the proposition. We have 
\[  \text{Quo}(\Gl)=\text{Quo}(\varprojlim \Gln)=\Union\text{Quo}(\Gln) .\] Now any element of $\text{Quo}(\Gln)$ must appear as one of the factor groups in a Jordan-H\"{o}lder series of $\Gln$. However, as is well known, the only (potentially) simple factor group that appears in a Jordan-H\"older series of $\Gln$ is $\PSL_{2}(\Fl)$ (see \cite[IV-25]{Se68}, for example). Then $\text{Quo}(\Gln)\subseteq\{[\PSL_{2}(\Fl)]\}$, where the brackets denote isomorphism class. Since $\PSL_{2}(\Fl)\not\simeq\PSL_{2}(\FF_{\ell'})$ for $\ell\ne \ell'$, we have $\text{Quo}(\Gl)\cap\text{Quo}(\GL_{2}(\ZZ_{\ell'})=\emptyset$.

\end{proof}
\subsection{The abelianization of $\text{GL}_2(\hat{\mathbb{Z})}$} 
Theorem~\ref{T:H=G} follows easily from Corollary~\ref{C:Max} once we have identified $\GZ^{\ab}=\GZ/(\GZ)'$. From the product description $\GZ=\Gprod$, we see immediately that $\GZ'=\prod_{\text{$\ell$ prime}}\Gl'$. So our task is reduced to determining $\Gl'$ for each prime $\ell$. 
\begin{lemma}Let  $\ell\ne 2$ be prime. Then $\Gl'=\SL_{2}(\Zl)=\ker(\Gl\xrightarrow\det\Zl^{*})$. 
\end{lemma}
\begin{proof}See \cite{LT}, Part II, \S 3, Lemma 1 and Part III, \S 4.
\end{proof}
The $\ell=2$ case is slightly subtler.  Recall first that we may identify
$\GL_2(\FF_2)$ with the permutation group $\mathfrak{S}_3$ by considering the 
matrices  as permutations of the three nonzero vectors of $\FF_2\times\FF_2$. This allows us to define a sign map $\sgn:\GL_2(\FF_2)\rightarrow\{\pm 1\}$. By composing with reduction maps, we get sign maps from $\GL_{2}(\ZZ_{2})$ and $\GZ$. By abuse of notation we will denote all of these maps by `$\sgn$'. 

\begin{lemma} The map $(\sgn,\det)\colon\GL_{2}(\ZZ_{2})\rightarrow\{\pm 1\}\times\ZZ_{2}^{*}$ is surjective. We have \[\GL_{2}(\ZZ_{2})'=(\ker\sgn)\cap\SL_{2}(\ZZ_{2})=\ker(\xymatrix{\GL_{2}(\ZZ_{2})\ar[r]^{(\sgn,\det)}&\ \{\pm 1\}\times\ZZ_{2}^{*}}).\]
\end{lemma}
\begin{proof}See \cite{LT}, Part III, \S 2. 
\end{proof}
Combining the two lemmas yields:
\begin{proposition}\label{P:Gab} The map $(\sgn,\det)\colon\GZ\rightarrow\Gab$ is surjective. The commutator subgroup of $\GZ$ is $\GZ'=\ker(\sgn,\det)$. We may identify the abelianization $\GZ\rightarrow\GZ^{\ab}$ with 
\[
\xymatrix{\GZ\ar [r]^{(\sgn,\det)}&\ \Gab}
\]
\end{proposition}
We can now prove our first theorem. 
\begin{proof}[Proof of Theorem~\ref{T:H=G}] If $H=\GZ$, then conditions (i) and (ii) obviously hold. Suppose $H\subsetneq\GZ$ and $\pi_{\ell}(H)=\Gl$ for all primes $\ell$. Then there is a maximal closed subgroup $K$ with $H\subseteq K\subsetneq G$.   Clearly $K$ also satisfies $\pi_{\ell}(K)=\Gl$ for all prime $\ell$. Then $K$ contains the commutator subgroup $\GZ'=\ker(\sgn,\det)$, by Proposition~\ref{P:Max}. Since $K\ne\GZ$, we have $(\sgn,\det)(K)\ne\Gab$. Since $H\subseteq K$, we also have $(\sgn,\det)(H)\ne \Gab$. 
\end{proof}
\subsection{Maximal closed subgroups of $\text{GL}_{2}(\hat{\ZZ})$}
It will be useful in what follows to have a more detailed picture of the maximal closed subgroup structure of $\GZ$.  According to Propositions~\ref{P:Max} and \ref{P:Gab}, we may proceed by examining the maximal closed subgroups of $\Gl$ and $\GZ^{\ab}\simeq\Gab$. 

For the most part we will be concerned with maximal closed subgroups $H\subsetneq\GZ$ for which the determinant map is surjective. Of course, maximal closed subgroups with $\det(H)\ne\Zhat^{*}$ correspond to maximal closed subgroups of $\Zhat^{*}$. These in turn are neatly described by class field theory via the isomorphism $\Zhat^{*}\simeq \Gal(\QQ^\ab/\QQ)$.

\subsubsection{Maximal closed subgroups arising from $\Gab$} \label{SSS:Gab} Let $H\subsetneq\GZ$ be a maximal closed subgroup such that $H_\ell=\Gl$ for all $\ell$ and $\det(H)=\Zhat^{*}$. By Corollary~\ref{C:Max} and the definition of $(\sgn,\det)$, this $H$ corresponds to a maximal subgroup $\Gab$ that surjects onto the two factors $\{\pm 1\}$ and $\Zhat^*$. It follows easily that the corresponding subgroup is the kernel of a character $\Gab\rightarrow\{\pm 1\}$ of the form $(\id,\chi)$, for some nontrivial character $\chi\colon\Zhat^*\colon\rightarrow\{\pm 1\}$. In other words, our original $H\subsetneq\GZ$ is the kernel of a character of the form $\sgn\cdot(\chi\circ\det)$ for some nontrivial character $\chi\colon\Zhat^*\rightarrow\{\pm 1\}$; that is $H=H_\chi:=\{g\in\GZ\colon \sgn(g)=\chi(\det(g))\}$. We call $H_\chi$ the \emph{Serre subgroup of $\GZ$ with character $\chi$}. 


\subsubsection{Maximal closed subgroups arising from $\text{GL}_{2}(\ZZ_{\ell})$.}\label{SSS:Gl} Suppose now that our maximal closed subgroup corresponds to a subgroup $H\subsetneq\Gl$. Set $M:=\ML_{2}(\ZZ_{\ell})$. The open normal subgroups $V_{\ell^{n}}:=I+\ell^{n}M$ constitute a fundamental basis of open neighborhoods of the identity in $\Gl$. For $n\geq 1$ the quotient $V_{\ell^n}/V_{\ell^{n+1}}$ is isomorphic to $\ML_2(\Fl)$, and comes equipped with a $\GL_2(\Fl)$-module structure; multiplication by $g\in\GL_2(\Fl)$ is defined as $g\cdot (I+\ell^nA):=I+\ell^nGAG^{-1}$, where $G$ is any lift of $g$ to $\GL_2(\ZZ/\ell^{n+1}\ZZ)$.  Now since $H$ is open, it must contain $V_{\ell^{n}}$ for some $n$, in which case $H$ corresponds to the maximal subgroup $H(\ell^{n})\subsetneq\Gln$. How big must $n$ be before we can see this correspondence? This question is answered by the following lemmas and corollaries.
\begin{lemma}{\cite[Part I, \S 6, Lemmas 2-3]{LT}} \label{L:ref1} Let $U\subseteq V_{\ell}=I+\ell\ML_{2}(\Zl)\subseteq\Gl$. 
\begin{itemize}
\item[(i)]If $\ell$ is odd and $U\surjects V_{\ell}/V_{\ell^{2}}$, then $U=V_{\ell}$.
\item[(ii)]If $\ell=2$, and $U\cap V_{4}\surjects V_{4}/V_{8}$, then $U\cap V_{4}=V_{4}$. If in addition $U\surjects V_{2}/V_{8}$, then $U=V_{2}$. 
\end{itemize}
\end{lemma}
\begin{lemma}{\cite[IV-23]{Se68}}\label{L:ref2} Let $\ell\geq 5$. Suppose $H\subseteq\Sl$ is a closed subgroup such that $H\surjects\SL_{2}(\Fl)$. Then $H=\Sl$. 
\end{lemma} 
\begin{corollary}\label{C:ref}Let $H\subseteq \Gl$ be a closed subgroup. 
\begin{itemize}
\item[(i)] If $\ell=2$ and  $H\surjects\GL_{2}(\ZZ/8\ZZ)$, then $H=\Gl$.
\item[(ii)] If $\ell$ is odd and $H\surjects\GL_{2}(\ZZ/\ell^2\ZZ)$, then $H=\Gl$. 
\item[(iii)] If $\ell\geq 5$, $H\surjects\Gfl$ and $\det(H)=\ZZ_\ell^{*}$, then $H=\Gl$. 
\end{itemize}
\end{corollary}
\begin{proof}The first two statements are simple consequences of Lemma~\ref{L:ref1} and the observation that if $H\surjects\GL_{2}(\ZZ/\ell^{n}\ZZ)\simeq\Gl/V_{\ell^{n}}$, then $(H\cap V_{\ell^{r}})\surjects V_{\ell^{r}}/V_{\ell^{n}}$ for any $r<n$. 

To prove the third statement, we need only show that $\SL_{2}(\Zl)\subseteq H$. Since $H\surjects\Gfl$, we also have $H'\surjects\Gfl'=\SL_{2}(\Fl)$. Then $H'\subseteq\Gl'=\Sl$ is a closed subgroup of $\Sl$ which surjects onto $\SL_{2}(\Fl)$.  Thus $H'=\Sl$, by Lemma~\ref{L:ref2}, and we see that $\Sl\subseteq H$, as desired.
\end{proof}   
\begin{corollary}\label{C:MaxClass} The maximal closed subgroups of $H\subsetneq\Gl$ are in 1-1 correspondence with 
\begin{itemize}
\item[(i)] the maximal subgroups of $\GL_{2}(\ZZ/8\ZZ)$, if $\ell=2$;
\item[(ii)] the maximal subgroups of $\GL_{2}(\ZZ/\ell^2\ZZ)$, if $\ell$ is odd.
\end{itemize}
For $\ell\geq 5$ the maximal closed subgroups of $\Gl$ with surjective determinant are in 1-1 correspondence with the maximal subgroups of $\Gfl$ with surjective determinant.

\end{corollary}
The maximal subgroups structure of $\Gfl$ for $\ell$ prime is well-known (See \cite[\S 2.6]{Se72} or \cite[p.36]{Ma77}, for example.) According to the corollary, for $\ell\geq 5$ these account for all maximal closed subgroups of $\Gl$ with surjective determinant. For the primes 2 and 3, we get a few extra closed subgroups coming from $\GL_{2}(\ZZ/8\ZZ)$ and $\GL_{2}(\ZZ/9\ZZ)$, respectively. We conclude this section with a slightly closer look at the subgroup structure of $\GL_{2}(\ZZ/8\ZZ)$.
\begin{lemma} Let $H$ be a subgroup of $\GL_{2}(\ZZ/8\ZZ)$ such that $H\surjects\GL_{2}(\ZZ/4\ZZ)$. Then $[G:H]\leq 2$. 
\end{lemma}
\begin{proof} Set $M:=\ML_2(\ZZ/8\ZZ)$. Since $H(I+4M)=\GL_2(\ZZ/8\ZZ)$, and since $\#(I+4M)=2^4$, we need only show that $\#(H\cap (I+4M))\geq 2^3$. For this it suffices to show that $H\cap (I+4M)\supseteq\{I+4A\colon \tr A \equiv 0\pmod 2\}$. As above, $I+4M$ is a $\GL_2(\FF_2)$-module, where the action is defined by conjugation. Since $H\surjects\GL_2(\FF_2)$, the subgroup $H\cap (I+4M)\subseteq I+4M$ is in fact a $\GL_2(\FF_2)$-submodule of $I+4M$. Furthermore $\{I+4A\colon \tr A \equiv 0\pmod 2\}$ is generated as a $\GL_2(\FF_2)$-module by $I+4\begin{pmatrix}0&1\\0&0\end{pmatrix}$. Thus we need only show that $I+4\begin{pmatrix}0&1\\0&0\end{pmatrix}\in H$. Since $H\surjects\GL_2(\ZZ/4\ZZ)$, it contains an element of the form $B=(I+2\begin{pmatrix}0&1\\0&0\end{pmatrix})(I+4A)$. Then $H$ also contains $B^2=I+4\begin{pmatrix}0&1\\0&0\end{pmatrix}$.
\end{proof}
\begin{corollary}\label{C:G8} Let $H\subseteq\GL_2(\ZZ_2)$ be a closed subgroup such that $H\surjects\GL_2(\ZZ/4\ZZ)$ and $(\sgn,\det)(H)=\{\pm 1\}\times \ZZ_2^*$. Then $H=\GL_2(\ZZ_2)$. 
\end{corollary}
\begin{proof} We need only prove that the mod 8 image $H(8)$ is all of $\GL_2(\ZZ/8\ZZ)$. By the lemma $H(8)$ is at most of index 2. Then $H(8)$ contains $\ker(\sgn,\det)$, the commutator of $\GL_2(\ZZ/8\ZZ)$, and corresponds via $(\sgn,\det)$ to a subgroup of $\{\pm 1\}\times (\ZZ/8\ZZ)^*$. But by hypothesis $(\sgn,\det)(H(8))=\{\pm 1\}\times (\ZZ/8\ZZ)^*$. Thus $H(8)=\GL_2(\ZZ/8\ZZ)$ and $H=\GL_2(\ZZ_2)$.
\end{proof}
\begin{remark} In fact, there are exactly seven index 2 subgroups of $\GL_2(\ZZ/8\ZZ)$, corresponding to the seven nontrivial characters of $\{\pm 1\}\times\ZZ/8\ZZ^*$. Let us denote the three nontrivial characters of $(\ZZ/8\ZZ)^*$ by $\chi_3$, $\chi_5$ and $\chi_7$; here $\chi_i$ is the unique character whose kernel is generated by $i$ in $(\ZZ/8\ZZ)^*$. Then the index 2 subgroups of $\GL_2(\ZZ_2)$ are the kernels of the characters  $\sgn$, $\chi_i\circ\det$ and $\sgn\cdot(\chi_i\circ\det)$, where $i\in\{3,5,7\}$. 

Suppose $H$ is one of these index 2 subgroups. Then the image of $H$ in $\GL_2(\ZZ/4\ZZ)$ is either all of $\GL_2(\ZZ/4\ZZ)$ or of index 2. Furthermore, the image is of index 2 if and only if $(I+4M)\subseteq H$. The only subgroups above for which this is true are $\ker(\sgn)$, $\ker(\chi_5\circ\det)$ and $\ker(\sgn\cdot(\chi_5\circ\det))$. Their corresponding images mod 4 are the three subgroups of $\GL_2(\ZZ/4\ZZ)$ of index 2: namely, $\ker(\sgn)$, $\ker(\det)=\SL_2(\ZZ/4\ZZ)$ and $\ker(\sgn\cdot\det)$. 

 \end{remark}

\section{Some arithmetic}
\subsection{The adelic representation} We return to the situation of an elliptic curve $E/K$ with $K$ a number field and consider its $\ell$-adic representations $\rho_{E,\linfty}\colon G_{K}\rightarrow\GZ$, and adelic representation $\rho_{E}\colon G_{K}\rightarrow\GZ$. Deriving necessary and sufficient conditions for $\rho_{E}$ to be surjective is now simply an exercise of translating the statements of Theorem~\ref{T:H=G} into statements about our Galois representations. 

\begin{theorem}\label{T:rhoSurj}
Let $E/K$ be an elliptic curve defined over a 
number 
field 
$K$. Let $\Delta\in K^\times$ be the discriminant of any Weierstrass model of $E/K$. 
Then $\rho_{E}$ is surjective if and only if
\begin{itemize}
\item[(i)] the $\ell$-adic representation $\rho_{\linfty}\colon G_K\rightarrow\Gl$ is surjective for all $\ell$,
\item[(ii)]$K\cap\QQ^{\cyc}=\QQ$ and
\item[(iii)] $\sqrt{\Delta}\notin K^{\cyc}$.
\end{itemize}
\end{theorem}
\begin{proof}
Set $H=\rho_{E}(G_{K})$. According to Theorem~\ref{T:H=G}, we have $H=\GZ$ if and only if $\pi_{\ell}(H)=\GZ$ for all $\ell$ and $(\sgn,\det)(H)=\Gab$. 

Since $\rho_{E,\linfty}=\pi_{\ell}\circ\rho_{E}$, the first statement is clearly equivalent to condition (i) above. It remains to show that the surjectivity of $(\sgn,\det)|_{H}$ is equivalent to conditions (ii) and (iii). To do so, we must understand how $\sgn$ and $\det$ arise from the arithmetic of our elliptic curve. 

The $\det$ map is easy to identify. From properties of the Weil pairing, it follows that it is essentially the cyclotomic character; i.e., we have a commutative diagram
\[\xymatrix{G_K\ar[r]^{\rho_{E}} \ar[dr]^{\text{res}}&\GZ\ar[d]^{\det}\\ 
&\GalKcyc\simeq\Zhat^*.
}\]

The $\sgn$ map, on the other hand, was defined as the composition
\[\GZ\xrightarrow{r_{2}}\GL_{2}(\FF_{2})\simeq\mathfrak{S_{3}}\xrightarrow{\sgn}\{\pm 1\} 
\]
Since $r_{2}\circ\rho_{E}=\rho_{E,2}$, if we start with a $\sigma\in G_{K}$, we see that $\sgn(\rho(\sigma))$ is $\pm 1$ depending on whether $\sigma$ is an even or odd permutation of the three nontrivial points of $E[2](\Kbar)$. If we choose a Weierstrass model for $E/K$ and write $e_{i}$ for the $x$-coordinates of the three nontrivial 2-torsion points, we have $\sqrt{\Delta}=\pm 4\prod_{i>j}(e_{i}-e_{j})$ (see \cite[\S 5.3]{Se72}).  Thus $\sigma$ is even if and only if $\sigma(\sqrt{\Delta})=\sqrt{\Delta}$. In other words, $\sgn\circ\rho_{E}=\chi_{\Delta}$, where $\chi_{\Delta}\colon G_{K}\rightarrow\{\pm 1\}$ is the (possibly trivial) character defined by $K(\sqrt{\Delta})$. 

Now consider the tower of fields
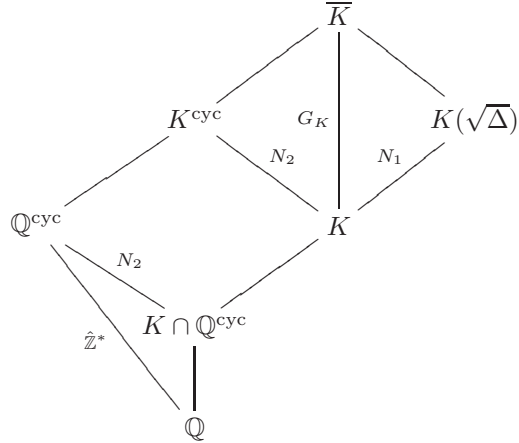
\[\xymatrix{&&\Kbar&\\
&K^{\cyc}\ar@{-}[ur]&&K(\sqrt{\Delta}) \ar@{-}[ul]\\
\QQ^{\cyc}\ar@{-}[ur]&&K\ar@{-}[ul]_{N_{2}}\ar@{-}[ur]^{N_{1}}\ar@{-}[uu]^{G_K}\\
&K\cap\QQ^{\cyc}\ar@{-}[ul]_{N_{2}} \ar@{-}[ur]\\
&\QQ\ar@{-}[uul]^{\Zhat^*}\ar@{-}[u] }
\]Here various Galois extensions have been labeled with their corresponding Galois 
group. 
Namely, we have (taking some liberties with identifications) 
$\Gal(\QQ^{\cyc}/\QQ)=\Zhat^*$, $\Gal(K(\sqrt{\Delta})/K)=N_{1}\subseteq\{\pm 1\}$ and 
$\Gal(\QQ^{\cyc}/K\cap\QQ^{\cyc})=\Gal(K^{\cyc}/K)= 
N_{2}\subseteq\Zhat^*$.

We have just seen that the map  
$(\sgn, 
\det)\circ\rho_{E}\colon G_{K}\rightarrow\Gab$ is just the product of the restriction maps
\[\xymatrix{G_K\ar[r]^-{\text{res}\times\text{res}}&N_{1}\times N_{2} \\
\sigma\ar@{|->}[r]& (\sigma\vert_{K(\sqrt{\Delta})},\sigma\vert_{K^{\cyc}}),}\]
and in general we have $(\sgn,\det)(H)\subseteq N_{1}\times N_{2}\subseteq\Gab$. Thus $(\sgn,\det)(H)=\Gab$ if and only if both set inequalities in this chain are in fact equalities. By Galois theory, the first inequality is an equality if and only if $\sqrt{\Delta}\notin K^{\cyc}$, and the second inequality is an equality  if and only if $\sqrt{\Delta}\notin K$ and $K\cap\QQ^{\cyc}=\QQ$. Take together, we conclude that $(\sgn,\det)(H)=\Gab$ if and only if $\sqrt{\Delta}\notin K^{\cyc}$ and $K\cap\QQ^{\cyc}=\QQ$.
\end{proof}
\begin{remark}\label{remark}Conditions (ii) and (iii) are equivalent to the single statement:
\begin{itemize}
\item[(ii)'] $K(\sqrt{\Delta})\cap\QQ^{\cyc}=\QQ$. 
\end{itemize}
Though this has the advantage of brevity, we prefer the stated form of the theorem as it more clearly points the way to finding elliptic curves with surjective adelic representations.  
\end{remark}
\begin{remark}The theorem and its proof elucidate  what happens when $K=\QQ$. Since $\QQ^{\cyc}=\QQ^{\ab}$, we have $\QQ(\sqrt{\Delta})\subseteq\QQ^{\cyc}$. Tracing through the various maps, we see that for any $\sigma\in G_{\QQ}$,  
\begin{align*}
\sgn(\rho_{E}(\sigma))&=\sigma\vert_{\QQ(\sqrt{\Delta})}\\&=(\sigma\vert_{\QQ^{\cyc}})\vert_{\QQ(\sqrt{\Delta})}\\&=\chi_{\Delta}(\det(\rho_{E}(\sigma))),
\end{align*}
where as before $\chi_{\Delta}\colon\Zhat^{*}\rightarrow\{\pm 1\}$ is the (possibly trivial) character arising from the extension $\QQ(\sqrt{\Delta})/\QQ$. Then $\rho_{E}(G_{\QQ})$ is contained in the Serre subgroup $H_{\chi_{\Delta}}=\{g\in\GZ\colon \sgn g=\chi_{\Delta}(\det g)\}$. Thus $[\GZ:\rho_{E}(G_{\QQ})]\geq[\GZ:H_{\chi_{\Delta}}]=2$. In particular, $\rho_{E/\QQ}(G_{\QQ})\ne\GZ$. 
\end{remark}

\subsection{Semistable elliptic curves}\label{Ksemistable}
Guided now by Theorem~\ref{T:rhoSurj}, we would like to find elliptic curves $E/K$ for which $\rho_{E, \linfty}$ is surjective for all $\ell$.  Recall that when $E/K$ is non-CM, the adelic image is open, which implies that $\rho_{E,\linfty}(G_K)=\Gl$ for all but finitely many primes.  Accordingly, we will call the primes $\ell$ for which $\rho_{E,\linfty}$ is not surjective the \emph{exceptional primes of $E/K$}.  Ideally we would like to be able to determine the set of exceptional primes for any given non-CM elliptic curve. For $\ell\geq 
5$, Corollary~\ref{C:ref} and the surjectivity of 
$\det:\rho_{E,\linfty}(G_K)\rightarrow\Zl^*$ 
imply that  $\rho_{E,\linfty}$ is surjective if and only if $\rho_{E,\ell}$ is 
surjective. For $\ell=2,3$  we have to do a little more work.

In either case, an important first step is to determine the mod $\ell$ image $\rho_{E,\ell}(G_K)$ for all $\ell$. It turns out that we can learn a lot about $\rho_{E,\ell}(G_K)$ simply by studying the image of inertia $\rho_{E,\ell}(I_w)$ for various inertia subgroups $I_w\subseteq\GalK$. (See Section~\ref{SS:Notations} for notations and definitions related to inertia groups.) Serre studies inertia representations extensively in \cite{Se72}. When the non-CM elliptic curve $E$ is semistable the results are particularly nice, yielding techniques for computing the exceptional primes of $E$.  Modulo some group theory, everything follows from the picture of the inertia representations given by the lemma below, which is essentially a synthesis of various facts scattered throughout \cite{Se72}: more specifically, the corollary to Proposition~13 in \S 1.12, and some properties of semistable curves discussed in \S 5.4.  

\begin{lemma}\label{L:Tate}Let $K$ be a number field, $\ell$ a rational prime unramified in $K$, and $E/K$ a  semistable elliptic curve with $j$-invariant $j_E$. Fix $v\in\Sigma_K$ and $w\in\Sigma_{\Kbar}$ with $w\mid v$. Recall that $S_E$ is the set of bad places of $E/K$, and that $S_\ell$ is the set of places $v\in\Sigma_K$ such that $v\mid l$. 
\begin{itemize}
\item[(i)] If $v\in\Sigma_K-S_E-S_\ell$, then $\rho_{E,\ell}(I_w)$ is trivial.
\item[(ii)]If $v\in S_E-S_\ell$, then $\rho_{E,\ell}(I_w)$ is either trivial or cyclic of 
order $\ell$.
\item[(iii)]If $v\in S_E$ and $\ell\nmid v(j_E)$, then $\rho_{E,\ell}(I_w)$ contains an element of order 
$\ell$.
\item[(iv)]If $v\mid l$, then  
\[\rho_{E,\ell}(I_w)=\left\{\begin{pmatrix}s&0\\0&1\end{pmatrix}: 
s\in\Fl^*\right\}, or \]
    \[\rho_{E,\ell}(I_w)=\left\{\begin{pmatrix}s&t\\0&1\end{pmatrix}: s\in\Fl^*, 
t\in\Fl\right\}, \] when $E$ has (good) ordinary reduction or bad (multiplicative) reduction at $v$; and 
$\rho_{E,\ell}(I_w)$ is a nonsplit Cartan subgroup, when $E$ has (good) supersingular reduction at $v$. 
  \end{itemize}
\end{lemma}
Amazingly enough this simple description of the inertia representations imposes strict restrictions on nonsurjective mod $\ell$ representations arising from a semistable $E/K$. The propositions and corollaries that follow are for the most part straightforward generalizations of Serre's results in \cite[\S 5.4]{Se72}. We formulate them for a number field $K$ satisfying the following properties:
\begin{itemize}
\item[(i)] There is a real embedding $K\injects \RR$. This gives rise to a complex conjugation map $\sigma\in G_K$ satisfying $\sigma^2=1$ and $\det(\rho_{E,\ell}(\sigma))=-1$ for all $\ell\geq 3$. It follows that $\rho_{E,\ell}(\sigma)$ is diagonalizable in $\Gfl$ for all $\ell\geq 3$, with eigenvalues 1 and -1. 
\item[(ii)] The narrow class group $\mathcal{C}_K^\infty$ is trivial. Recall $\mathcal{C}_K^\infty$ is the group of fractional ideals of $K$ modulo the subgroup of totally real principal fractional ideals. This assumption has as a consequence that any abelian extension of $K$ unramified at all finite primes is trivial. 
\item[(iii)]  We have $K\intersect\QQ^\cyc=\QQ$. This property ensures that $\det\colon \rho_E(G_K)\rightarrow\Zhat^*$ is surjective. 

\end{itemize}

\begin{proposition}\label{semistableborel} Let $K$ be a number field with a real embedding, a trivial narrow class group, and satisfying $K\cap\QQ^\cyc=\QQ$. Let $E/K$ be a semistable elliptic curve with $j$-invariant $j_E$. Suppose $\ell$ is a 
prime unramified in $K$. If  $\ell=2,3,5$, suppose further that $\ell\nmid v(j_E)$ for some $v\in S_E$. If
$\rho_{E,\ell}(G_K)\ne\Gfl$, then $\rho_{E,\ell}(G_K)$ is contained in a Borel subgroup of $\Gfl$.
\end{proposition}
\begin{proof}
The proposition is nearly identical to Proposition 21 in \cite{Se72}. As such we are content to sketch a proof, mainly just to illustrate Lemma~\ref{L:Tate} at work. 

If  $v\in S_E$ and $\ell\nmid v(j_E)$, then according to Lemma~\ref{L:Tate}, the mod $\ell$ image contains an element of order $\ell$. From group theory it follows that the mod $\ell$ image either contains $\SL_2(\FF_\ell)$ or is contained in a Borel. The former is impossible as the determinant map is surjective (since $K\cap\QQ^\cyc=\QQ$), and we assume the mod $\ell$ representation is not surjective.  

Now assume $\ell$ is unramified in $K$ and $\ell\geq 7$. Lemma~\ref{L:Tate} implies the mod $\ell$ image contains a split semi-Cartan subgroup or a nonsplit Cartan subgroup. Again it follows from group theory that the mod $\ell$ image is contained in either a Borel subgroup, a Cartan subgroup, or else it is contained in the normalizer of a Cartan subgroup, but not the Cartan subgroup itself. The last case would give rise to a (nontrivial) unramified character $\chi\colon G_L\rightarrow\{\pm 1\}$, contradicting the fact that $K$ has trivial narrow class group. If the mod $\ell$ image is contained in a Cartan subgroup, it must be a split Cartan subgroup, thanks to the complex conjugation $\sigma\in G_K$, which is diagonalizable mod $\ell$. Since split Cartan subgroups are contained in a Borel subgroup, we are done.
\end{proof}

As we mentioned in the Introduction, Theorem~\ref{T:rhoSurj} leads the hunter of elliptic curves with surjective adelic representations naturally to non-Galois cubic extensions of $\QQ$. With this in mind we include the following corollaries, which specialize to number fields $K$ with $[K:\QQ]=3$. Note that in this case the existence of a real embedding is automatic.    
\begin{corollary}\label{C1:semistable}Let $E, K$ and $\ell$ be as in 
Proposition~\ref{semistableborel} 
and 
suppose that $\rho_{E,\ell}(G_K)\ne\Gfl$. Assume further that $[K:\QQ]=3$ and and that $(U_K^+-1)\cap U_K\ne\emptyset$. There is a basis of $E[l](\Kbar)$ in terms of 
which 
$\rho_{E,\ell}$ is of the form $\begin{pmatrix}\chi_1&*\\0&\chi_2\end{pmatrix}$  for 
characters 
$\chi_i:G_K\rightarrow \Fl^*$. Furthermore one of the characters is trivial and  the other is $\det\circ\rho_{E,\ell}$.
\end{corollary}
\begin{remark} Recall that  $U_K$ (resp. $U_K^+$) is the group of units (resp. totally positive units) of $K$.
\end{remark}
\begin{proof}Since $\rho_{E,\ell}(G_K)\ne\Gfl$, Proposition~\ref{semistableborel} 
implies 
$\rho_{E,\ell}(G_K)$ is contained in a Borel subgroup. The first statement now follows 
easily.

Assume we have picked a basis so that $\rho_{E,\ell}$ is of the form $\begin{pmatrix}\chi_1&*\\0&\chi_2\end{pmatrix}$. Since 
$\chi_1\cdot\chi_2=\det\circ\rho_{E,\ell}$, we need only show that one of the characters is trivial. A character $\chi\colon G_{K}\rightarrow\Fl^{*}$ is trivial if and only if it is unramified for all $v\in\Sigma_{K}$: a consequence of $K$ having trivial narrow class group. Thus we need only show that one of the two characters is unramified everywhere. 

First observe that both characters are unramified for all $v\nmid l$. Indeed, if $v\notin S_E$ and $v\nmid l$, then $\rho_{E,\ell}$ is itself unramified.  Likewise, if $v\in S_E$ and $v\nmid l$, then by 
Lemma~\ref{L:Tate} for any $w\mid v$ the image of $I_w$ in $\Gfl$ is either 
trivial or cyclic of order $\ell$. In either case, we see that
\[ \rho_{E,\ell}(I_w)\subseteq\left\{\begin{pmatrix}1&t\\0&1\end{pmatrix}: 
t\in\FF_\ell\right\}, \]
whence both $\chi_i$ are unramified.  So it only remains to show that there is one character that is also unramified at each place $v\mid l$. The argument now divides into cases depending on the splitting behavior of $\ell$.

\noindent\emph{Case 1: $\ell$ is inert.} Take the unique $v\mid l$ and an inertia group $I_w$ for some $w\mid v$. The image of  inertia $\rho_{E,\ell}(I_w)$ cannot be a nonsplit Cartan subgroup as it is contained in a Borel subgroup.  But then by Lemma~\ref{L:Tate}, $\rho_{E,\ell}(I_w)$ must be of the form 
$\begin{pmatrix}*&0\\0&1\end{pmatrix}$ or 
$\begin{pmatrix}*&*\\0&1\end{pmatrix}$.  Then 
one of the $\chi_i$, call it $\chi_{i_0}$, is trivial when restricted to $I_w$. This shows that $\chi_{i_0}$ is unramified at $v$, and hence everywhere, as desired.

\noindent\emph{Case 2: $\ell$ is totally split.} Suppose $(\ell)=\mathfrak{p}\mathfrak{q}\mathfrak{r}$. As in the 
inert case, at each $v\mid l$, exactly one of the characters is unramified. Since there are three places above $\ell$, by the pigeonhole principle one of the characters, call it $\chi_{i_0}$, is unramified 
at at least two of the places.  

Suppose $\chi_{i_0}$ is ramified at exactly 
one place. Assume this place is $v=\mathfrak{p}$.  In terms of Galois theory, $\chi_{i_0}$ 
corresponds to an abelian extension $L/K$ with $\Gal(L/K)\simeq \Fl^*$ such that only 
$\mathfrak{p}$ and possibly $\infty$ ramify in $L$. According to class field theory, there is a 
modulus of the form $\mathfrak{m}=\infty\cdot\mathfrak{p}^n$ such that $L$ is contained in the 
ray class field $K_{\mathfrak{m}}$. We then have a surjection 
$\mathcal{C}_K^{\mathfrak{m}}\simeq\Gal(K_{\mathfrak{m}}/K)\surjects 
\Gal(L/K)\simeq\Fl^*$, 
where $\mathcal{C}_K^{\mathfrak{m}}$ is the group of fractional deals of $K$ relatively prime to 
$\mathfrak{p}$ modulo the group of principal ideals of the form $(a)$ where 
$a\equiv 1 
\pmod{\mathfrak{p}^n}$ and $a$ is totally positive. Furthermore there is an exact sequence (\cite[\S VI.1]{Neu})
\[1\rightarrow U_K^+/U_{\mathfrak{m},1}\rightarrow 
(\OK/\mathfrak{p}^n)^*\rightarrow 
\mathcal{C}_K^{\mathfrak{m}}\rightarrow\mathcal{C}_K^{\infty}\rightarrow 1 \]
where
$U_{\mathfrak{m},1}$ is the 
subgroup of totally positive units which are congruent to 1 modulo 
$\mathfrak{p}^n$.  Since  $\mathcal{C}_K^{\infty}=1$ in our case, we get a composition of surjections
\[ (\OK/\mathfrak{p}^n)^*\surjects\mathcal{C}_K^{\mathfrak{m}}\surjects \Fl^*, \]
whose kernel contains $U_K^+/U_{\mathfrak{m},1}$. As $\ell\nmid (\ell-1)$, the composition must  factor as
\[ 
\SelectTips{eu}{}\xymatrix@C=1pt{(\OK/\mathfrak{p}^n)^*\ar@{->>}[rr]\ar@{->>}[dr
]& 
&\Fl^*\\ &(\OK/\mathfrak{p})^*\ar@{->>}[ur]& }.\] Since  
$(\OK/\mathfrak{p})^*\simeq\Fl^*$, 
the surjection $(\OK/\mathfrak{p})^*\surjects\Fl^*$ is in fact an isomorphism.

Now take any $u\in (U_K^+-1)\cap U_K$. Then $u$ is a unit and $u+1\in U_K^+$. As the image of $u+1$ in $(\OK/\mathfrak{p})^*$ is in the kernel of the isomorphism $(\OK/\mathfrak{p})^*\rightarrow \Fl^*$, 
we must 
have 
$u+1\equiv 1 \pmod{\mathfrak{p}}$. But then $u\equiv 0\pmod{\mathfrak{p}}$, a 
contradiction as $u$ is a unit. Thus $\chi_{i_{0}}$ must be ramified at all places in $S_{\ell}$, and hence at all places in $\Sigma_{K}$. It follows that $\chi_{i_{0}}$ is trivial. 

\noindent{\emph{Case 3: $(\ell)=\mathfrak{p}\mathfrak{q}$.}} Lastly, suppose $(\ell)=\mathfrak{p}\mathfrak{q}$, with $f(\mathfrak{p}):=[\mathcal{O}_K/\mathfrak{p}\mathcal{O}: \Fl]=2$. Assume each character is ramified at exactly one of the primes 
lying above $\ell$. 
Suppose $\chi_{i_0}$ is ramified at $\mathfrak{q}$ and $\chi_{1-i_0}$ is 
ramified at 
$\mathfrak{p}$. Then, using $\chi_{i_0}$, we may argue exactly as in the totally 
split 
case 
to show that $\alpha\in\mathfrak{q}$, a contradiction. Thus one of the 
characters is 
unramified at both primes lying above $\ell$, making it trivial.
\end{proof}
\begin{corollary}
\label{C2:semistable}
Let $E, K$ and $\ell$ be as in Corollary~\ref{C1:semistable} and $\rho_{E,\ell}(G_K)\ne\Gfl$. Given $v\in\Sigma_K-S_E$, let $\phi_v\in\End(\tilde{E_v})$ be the Frobenius endomorphism and let $t_v$ be its trace. Then $t_v\equiv 1+N_v\pmod l$.
\end{corollary}
\begin{remark} Since 
$\#\tilde{E_v}(k_v)=1-t_v+N_v$, the condition $t_v\equiv 
1+N_v\pmod l$ 
is equivalent to $\ell\mid\#\tilde{E_v}(k_v)$.
\end{remark}
\begin{proof}Suppose first that $v\in \Sigma_K-S_E-S_\ell$. The representation $\rho_{\linfty}$ is unramified at 
$v$ and the $\ell$-adic Tate modules of $E/K$ and its reduction $\tilde{E_v}/k_v$ are isomorphic as $D_w/I_w$-modules for any $w\in S_v$.  Then $\tr(\phi_v)=\tr(\rho_{\ell}(\Frob_w))\pmod l$ and 
$N_v=\det(\phi_v)=\det (\rho_{\ell}(\Frob_w))\pmod l$ for any $w\in S_v$. (Observe that although strictly speaking $\Frob_w$ is a coset in $D_w/I_w$, the value $\rho_{\ell}(\Frob_w)$ is well-defined, as $\rho_\ell$ is unramified at $v$.)
Now by Corollary~\ref{C1:semistable},
\begin{eqnarray*}
t_v\equiv\tr(\rho_{\ell}(\Frob_w))&\equiv&\chi_1(\Frob_w)+\chi_2(\Frob_w)\\
&\equiv&1+\det(\rho_{E,\ell}(\Frob_w)) \\
&\equiv&1+\det (\rho_{\ell}(\Frob_w))\\
&\equiv&1+N_v \pmod l,\end{eqnarray*}
and the claim is proved in this case. 

Now suppose $v\notin S_E$ but $v\in S_\ell$. Since $\rho_{E,\ell}(G_K)$ is contained in a 
Borel 
subgroup, it cannot contain a nonsplit Cartan subgroup. It follows from 
Lemma~\ref{L:Tate} that $E$ has ordinary reduction at $v$.

First consider $\ell=2$. Let $v$ be a place of $K$ lying over 2. Since $E$ has good ordinary reduction at $v$, the reduction $\tilde{E_v}$ has exactly one point, $P$, of order 2. Then $P$ is fixed by $\Galkv$, hence $k_v$-rational. But then $2$ divides $\#\tilde{E}_v(k_v)=1-t_v+N_v$, in which case $t_v\equiv 1+N_v\pmod 2$.

Now consider $\ell\geq 3$. 
Pick a basis $\{P_1,P_2\}$ of $\El[\Kbar]$ so that 
$\rho=\begin{pmatrix}\chi_1&*\\0&\chi_2\end{pmatrix}$, as in 
Corollary~\ref{C1:semistable}. 
We know that one of the $\chi_i$ is trivial. 

Suppose $\chi_1=1$. Then $E$ has a $K$-rational point $P$ of order $\ell$. If 
$\langle 
P\rangle$ is in the kernel of the reduction map, we have an exact sequence
\[0\rightarrow \langle P\rangle\rightarrow E[l](\Kbar)\rightarrow\tilde{E_v}[l](\overline{k_v})\rightarrow 0.\]
But then the representation of $I_w$ for any $w\mid v$ looks like
$\begin{pmatrix}1&*\\0&1\end{pmatrix}$, contradicting Lemma~\ref{L:Tate}.
Thus the reduction map sends $P$ to a nontrivial $k_v$-rational point of 
$\tilde{E}_v[l](k_v)$. It follows that $\ell$ divides $\#\tilde{E}_v(k_v)$, whence $t_v\equiv 1+N_v\pmod l$. 

Suppose $\chi_2=1$. Let $C$ be the $G_K$-invariant cyclic subgroup
defined by $P_1$. Consider the quotient $E'=E/C$. Since $E'$ is isogenous to $E$, it has the same reduction type at all places of $\Sigma_K$, and furthermore $\rho_{E'}\sim\rho_E$. In particular, it follows that $t_v'=t_v$ and $\#\tilde{E'}_v(k_v)=\#\tilde{E}_v(k_v)$ for our place $v$.  Now since $\chi_2$ is trivial, $E'[l]$ has a 
nontrivial 
$K$-rational point, and we may argue as in the $\chi_1=1$ case to prove 
$t_v\equiv 
1+N_v\pmod l$.
\end{proof}

Suppose $K$ satisfies the conditions of the previous corollaries. We now have the necessary means for determining the set of primes $\ell$ for which $\rho_{E,\ell}$ is surjective for a given semistable elliptic curve $E/K$. First 
compute $\#\tilde{E}_v(k_v)$  for some $v\notin S_E$. Let $R$ be the set of prime divisors of $\#\tilde{E}_v(k_v)$ and let $T$ be the set of primes in $\ZZ$ that ramify in $K$. According  to Corollary~\ref{C2:semistable}, the set of primes $\ell$ for which $\rho_{E,\ell}$ is not 
surjective is 
contained in $\{2,3,5\}\union R\union T$. For this finite set of primes we can then 
use the following criterion for checking whether $\rho_{E,\ell}(G_K)=\Gfl$.
\begin{proposition}\label{frobdisc} Let $\ell\geq 5$, and suppose $H\subseteq\Gfl$ is 
a subgroup  satisfying
\begin{itemize}
\item[(i)] $H$ contains elements $s_1,s_2$ such that 
$(\frac{\tr(s_i)^2-4\det(s_i)}{l})=(-1)^i$ and $\tr(s_i)\ne 0$.
\item[(ii)]$H$ contains an element $t$ such that $u=\tr(t)^2/\det(t)\ne 0,1,2,4$ 
and 
$u^2-3u+1\ne 0$.
\end{itemize}
Then $H$ contains $\SL(\Fl)$. In particular, if $\det:H\rightarrow\Fl^*$ is 
surjective, 
then 
$H=\Gfl$.
\end{proposition}
\begin{proof}
See \cite[Prop. 19]{Se72}
\end{proof}
\subsection{A suitable cubic extension}\label{S:K}
Let us fix a suitable number field. For the remainder of the paper we will let $K$ be the cubic extension $\Qalpha$, where $\alpha$ is the real root of $f(x)=x^3+x+1$. 

We easily see that $K$ satisfies the conditions of Corollaries \ref{C1:semistable} and \ref{C2:semistable}. The root $\alpha$ defines the sole real embedding $K\injects\RR$. The discriminant of $f$ is $-31$. This implies that $K$ is non-Galois, and hence that $K\cap\QQ^\cyc=\QQ$. It also follows that the ring of integers $\mathcal{O}_K$ is $\ZZ[\alpha]$, and that $31$ is the only rational prime that ramifies in $\OK$. Further computation then reveals that the ideal and narrow class groups of $K$ are trivial. Lastly we show that $\alpha$ is an element of $(U_K^+-1)\cap U_K$. Since $\alpha(\alpha^2+1)=-1$, we have $\alpha\in U_K$. (In fact one can show that $\alpha$ generates $U_K$.) But then $\alpha+1=-\alpha^3$ is also a unit. It is also not difficult to see that $\alpha+1$ is positive, and hence totally positive. Thus we have $\alpha+1\in U_K^+$ and $\alpha\in(U_K^+-1)$.  

As described in Section~\ref{Ksemistable}, with the help of Corollaries \ref{C1:semistable} and \ref{C2:semistable} we can now easily find elliptic curves $E/\QQ(\alpha)$ with surjective adelic representations.


\subsection{An example}\label{SS:example}
Let $K=\QQ(\alpha)$ and let $E/K$ be the elliptic curve $y^2+2xy+\alpha y=x^3-x^2$. We compute 
$(\Delta_E)=P_{131}Q_{2207}$, where the rational primes 131 and 2207 factor as 
$(131)=P_{131}Q_{131}R_{131}$ and $(2207)=P_{2207}Q_{2207}$, with  
$f(P_{2207})=2$. 
Furthermore, $(j_E)=(2)^{12}(3)^3/Q_{131}Q_{2207}$. Since the conductor of an 
elliptic curve divides the discriminant (\cite[IV.11.2]{ATAEC}), we see that $E$ 
is semistable with conductor $N=P_{131}Q_{2207}$.

Set $H=\rho(G_K)\subseteq\GZ$. From 
the  splitting behavior of 131 and 2207 we may deduce that $\sqrt{\Delta}\notin K^{\cyc}$. Since in addition $K\cap\QQ^\cyc=\QQ$, it follows that the abelianization map $(\sgn,\det):H\rightarrow\Gab$ is surjective. By 
Theorem~\ref{T:rhoSurj} we need only show that $E/K$ has no exceptional primes; i.e., that  
$H_\ell=\Gl$ for all prime $\ell$.

Recall that for a good place $v\in S_E$, we denote by $t_v$ the trace of the Frobenius element $\phi_v\in\End(\tilde{E}_v)$. Using {\sc Magma} (\cite{MAGMA}) we now reduce at various places to obtain the following table.
\[
\begin{array}{lllllr}
v=(7)&\#\tilde{E}_v(k_v)=324& N_v=343&t_v=20&(t_v^2-4N_v^2)\equiv& 20\pmod{31}\\
v=Q_{11} &\#\tilde{E}_v(k_v)=16& N_v=11&t_v=-4&(t_v^2-4N_v^2)\equiv &3\pmod{31}\\
v=Q_{23} &\#\tilde{E}_v(k_v)=15& N_v=23&t_v=9\\
v=Q_{29} &\#\tilde{E}_v(k_v)=24& N_v=29&t_v=6.
\end{array}
\]

Since $v(j_E)=-1$ for all $v\in S_E$, it follows from 
Corollary~\ref{C2:semistable} that 
for 
all $\ell\ne 31$, if $H(\ell)\ne \Gfl$, then $\ell\mid 16$ and $\ell\mid 15$ (the values of 
$\#\tilde{E}_v(k_v)$ in 
rows 
2 and 3 of our table). There is no such $\ell$. Thus $H(\ell)=\Gfl$ for all $\ell\ne 31$. 
Since 
$\det_{H}$ is surjective, Corollary~\ref{C:ref} implies $H_\ell=\Gl$ for all $\ell\ne 
2,3,31$. It remains only to show that these three primes are not exceptional.
\\ \\
\noindent\emph{Case $\ell=31$.} The values (modulo 31) of $t_v^2-4N_v^2$ for $v=(7)$ and 
$v=Q_{11}$ are 
20 
and 3 respectively. The first is a square modulo 31; the second is not. 
Furthermore, for 
$v=(7)$ we have $u=t_v^2/N_v\equiv 10\not\equiv 0,1,2,4 \pmod{31}$, and 
$u^2-3u+1\not\equiv 
0 \pmod{31}$. Thus setting $s_1$ and $t$ equal to $\rho_{E,31}(\Frob_w)$ for any $w\mid (7)$, and setting 
$s_2$ equal to $\rho_{E,31}(\Frob_{w'})$ for any $w'\mid Q_{11}$, we see that 
$H(31)\subseteq\GL_2(\FF_{31})$ satisfies the conditions of 
Proposition~\ref{frobdisc}. 
Thus 
$H(31)$ contains $\SL_2(\FF_{31})$. Since $\det:H(31)\rightarrow\FF_{31}^*$ is 
surjective, 
we have $H(31)=\GL_2(\FF_{31})$, and hence $H_{31}=\GL_2(\ZZ_{31})$.
\\ \\
\noindent\emph{Case $\ell=3$.} Let $M:=\ML_2(\ZZ_3)$. Since $H(3)=\GL_2(\FF_3)$, we need only show that $H\supseteq I+3M$. By Lemma~\ref{L:ref1}, it suffices to show 
that $H(9)\supseteq(I+3M)/(I+9M)$. Let $v=Q_{29}$, 
and let 
$\pi\in H_3$ be a $\rho(\Frob_w)$ for any $w\in S_v$. From our table, the characteristic 
polynomial of 
$\pi$ is $t^2-6t+29$. Modulo 9 this factors as $(t-7)(t-8)$. Since $7\not\equiv 
8\pmod 
3$, 
$\pi$ is diagonalizable in $\GL_2(\ZZ_3)$. After a change of basis, we may 
assume that 
$\pi\equiv \begin{pmatrix}-2&0\\0&-1\end{pmatrix}\pmod 9$, in which case 
\[\pi^2\equiv\begin{pmatrix}4&0\\0&1\end{pmatrix}\equiv 
I+3\begin{pmatrix}1&0\\0&0\end{pmatrix}\pmod 9.\]
But $(I+3M)/(I+9M)$ is a $\GL_2(\FF_3)$-module, and since $H(9)\surjects\GL_2(\FF_3)$ it follows that $H(9)\cap(I+3M)/(I+9M)$ is a $\GL_2(\FF_3)$-submodule. (See \ref{SSS:Gl}.) Furthermore it is easily seen that $I+3\begin{pmatrix}1&0\\0&0\end{pmatrix}$ generates $(I+3M)/(I+9M)$ as a $\GL_2(\FF_3)$-module. Thus $H(9)\supseteq(I+3M)/(I+9M)$, and hence $H_3=\GL_2(\ZZ_3)$.  
\\ \\
\noindent\emph{Case $\ell=2$.} Let $M:=\ML_2(\ZZ_2)$. First we will show that $H(4)=\GL_2(\ZZ/4\ZZ)$. Since $H\surjects\GL_2(\FF_2)$, it suffices to show that $H(4)\supseteq (I+2M)/(I+4M)$.  

Let $\pi=\rho_{2^{\infty}}(\sigma)\in H_2$ be the image of a complex 
conjugation 
automorphism $\sigma\in G_K$. A calculation shows that $\Delta_E$ is positive 
(thinking 
of 
$K=\QQ(\alpha)$ as a subfield of $\RR$). Thus $\sqrt{\Delta_E}$ is fixed by 
complex 
conjugation. This means that $\pi\in\ker(H_2\xrightarrow{\sgn}\{\pm 
1\})=N(2^{\infty})$; 
i.e., the image $r_2(\pi)$ is contained in the normal subgroup
\[\left\{ 
I,\begin{pmatrix}1&1\\1&0\end{pmatrix},\begin{pmatrix}0&1\\1&1\end{pmatrix} 
\right\}\subseteq\GL_2(\FF_2).\]
But from the remarks in Section~\ref{S:K}, we have  $\tr\pi=1+(-1)=0$.  Thus 
$\pi\equiv 
I\pmod 2$; i.e., we have $\pi=I+2A\in I+2M$.
Since the characteristic polynomial of $\pi$ is $t^2-1$, it follows that the 
characteristic 
polynomial of $A$ is $t^2+t$. As this has distinct roots modulo 2, it follows 
that $A$, 
and 
hence $\pi$, is diagonalizable in $\GL_{2}(\ZZ_{2})$. After a suitable 
change of 
basis we may assume that 
\[\pi=\begin{pmatrix}1&0\\0&-1\end{pmatrix}
=I+2\begin{pmatrix}0&0\\0&-1\end{pmatrix}=:I+2A.\]

As with the $\ell=3$ case, since $H(2)=\GL_2(\FF_2)$, the subgroup $H(4)\cap(I+2M)/(I+4M)$ is in fact a $\GL_2(\FF_2)$-submodule of $(I+2M)/(I+4M)$. Again it is easily seen that $I+2A$ generates $(I+2M)/(I+4M)$ as a $\GL_2(\FF_2)$-module. Thus $H(4)\supseteq (I+2M)/(I+4M)$ and $H(4)=\GL_2(\ZZ/4\ZZ)$.

Since $(\sgn,\det)(H)=\{\pm 1\}\times\ZZ_2^*$ and $H(4)=\GL_2(\ZZ/4\ZZ)$, it now follows from Corollary~\ref{C:G8} that $H=\GL_2(\ZZ_2)$.

Having shown that $H_\ell=\Gl$ for all $\ell$, and that $(\sgn,\det)(H)=\Gab$, we conclude that $H=\GZ$. In other words, the adelic representation $\rho_E$ is surjective in this example. 
\bibliography{SurjAdelicRep}
\end{document}